\documentclass{amsart}
\usepackage{amsmath,amssymb,amscd,amsthm}
\numberwithin{equation}{section}
\newtheorem{theorem}[equation]{Theorem}
\newtheorem{lemma}[equation]{Lemma}

\newtheorem{proposition}[equation]{Proposition}
\newtheorem{definition}[equation]{Definition}
\theoremstyle{definition}
\newtheorem{remark}[equation]{Remark}
\newcommand\bG{{\mathbf G}}
\newcommand\bL{{\mathbf L}}
\newcommand\bS{{\mathbf S}}
\newcommand\bT{{\mathbf T}}
\newcommand\BC{{\mathbb C}}
\newcommand\BF{{\mathbb F}}
\newcommand\BG{{\mathbb G}}
\newcommand\BN{{\mathbb N}}
\newcommand\BQ{{\mathbb Q}}
\newcommand\BZ{{\mathbb Z}}
\newcommand\cP{{\mathcal P}}
\newcommand\fI{{\mathfrak I}}
\newcommand\GF{{\bG^F}}
\newcommand\SF{{\bS^F}}
\newcommand\SFl{{\bS^F_\ell}}
\newcommand\eps{{\varepsilon}}
\newcommand\vl{{v_\ell}}
\DeclareMathOperator\End{End}
\DeclareMathOperator\Gal{Gal}
\DeclareMathOperator\GL{GL}
\DeclareMathOperator\Hom{Hom}
\DeclareMathOperator\Id{Id}
\DeclareMathOperator\Irr{Irr}
\DeclareMathOperator\Ker{Ker}
\DeclareMathOperator\Res{Res}
\newcommand\Xbar{{\bar X}}
\newcommand\lexp[2]{\kern\scriptspace\vphantom{#2}^{#1}\kern-\scriptspace#2}
\newcommand\inv{^{-1}}
\newcommand\card[1]{|#1|}
\title{The Sylow subgroups of a finite reductive group}
\author{Michel Enguehard}
\email{michel.enguehard@imj-prg.fr}
\address{Universit\'e Paris-Diderot}
\author{Jean Michel}
\email{jean.michel@imj-prg.fr}
\address{Universit\'e Paris-Diderot}
\dedicatory{Dedicated to professor George Lusztig on the occasion of his 70th
birthday}
\date{22nd July 2016}
\begin{document}
\begin{abstract}
We  describe the structure of Sylow  $\ell$-subgroups of a finite reductive
group $\bG(\BF_q)$ when $q\not\equiv 0 \pmod \ell$ that we find governed by
a  complex reflection group attached to  $\bG$ and $\ell$, which depends on
$\ell$  only through the set of cyclotomic  factors of the generic order of
$\bG(\BF_q)$  whose value at  $q$ is divisible  by $\ell$. We also tackle
the  more general case of groups $\bG^F$  where $F$ is an  isogeny some 
power of which is a Frobenius morphism.
\end{abstract}
\maketitle
\section{Introduction}
\begin{definition}\label{finite reductive}
Let  $\bG$ be a connected reductive group over $\overline\BF_p$, and $F$ an
isogeny such that some power of $F$ is a Frobenius endomorphism; then $\GF$
is what we call a {\em finite reductive group}. To this situation we attach
a  positive real  number $q$  such that  for some  integer $n$, the isogeny
$F^n$ is the Frobenius endomorphism attached to a $\BF_{q^n}$-structure.
\end{definition}
The  goal of this note  is to describe the  Sylow $\ell$-subgroups of $\GF$
when  $\ell$ is  a prime  different from  $p$ and  $\bG$ is semisimple. The
structure  of the  Sylow $\ell$-subgroups  of a  Chevalley group  was first
described  by \cite{GL}  where they  observed that  they had a large normal
abelian   subgroup  $(\BZ/n)_\ell^a$  where  $n$   is  the  $\ell$-part  of
$\Phi_d(q)$,  where $d$ is  the multiplicative order  of $q \pmod\ell$, and
they computed $a$ case by case.

In  1992 \cite{BM}  exhibited subtori  of $\GF$  attached to eigenspaces of
elements  of the Weyl reflection coset of $(\bG,F)$ whose $F$-stable points
are  the  large  abelian  groups  of  \cite{GL}.  To  these eigenspaces are
attached complex reflection groups by Springer's theory.

We  show  that  the  structure  of  the  Sylow $\ell$-subgroups of $\GF$ is
determined  by these complex reflection groups. The results of this note in
the  case when $F$ is  a Frobenius were obtained  by the first author in an
unpublished  note \cite{En} of 1992; the  second author has found a simpler
(containing more casefree steps) proof which is an occasion to publish 
these results. Some of our results appeared also implicitly in \cite{malle}.

The  second author wishes to  thank Carles Broto for  a visit to Barcelona,
which started him thinking about this topic.

We thank Rapha\"el Rouquier for discussions which helped with the proofs of
Propositions \ref{q-indec} and \ref{gensylow}(4).

\section{The generic Sylow theorems}
Let  $\bG$ be  as in  \ref{finite reductive};  an $F$-stable  maximal torus
$\bT$  of  $\bG$  defines  the  Weyl  group $W=N_\bG(\bT)/\bT$, that we may
identify to a reflection subgroup of $\GL(X(\bT))$ where
$X(\bT):=\Hom(\bT,\BG_m)$,  attached to the root system
$\Sigma\subset  X(\bT)$  of  $\bG$  with  respect  to  $\bT$. The isogeny $F$
induces a $p$-morphism $F^*\in\End(X(\bT))$ by the formula $F^*(x)=x\circ F$
for  $x\in X(\bT)$, that is there is a
permutation  $\sigma$  of  $\Sigma$  such  that  for  $\alpha\in\Sigma$ we have
$F^*(\alpha)=q_\alpha\sigma(\alpha)$  for some power  $q_\alpha$ of $p$; in
particular  $F^*\in  N_{\End(X(\bT))}(W)$.  

If  $q,n$ are as in \ref{finite reductive}  then $F^{*n}$ is $q^n$ times an
element of $\GL(X(\bT))$ of finite order, thus over
$X(\bT)\otimes\BZ[q\inv]$   we   have   $F^*=q\phi$   where  $\phi$  is  an
automorphism of finite order which normalizes $W$.  We call $W\phi$ the {\em
reflection coset associated to $(\bG,F)$}.

Our  setting is more general than that of \cite{BM} who considered only the
special  cases where $F$ is  a Frobenius endomorphism, or  where $\GF$ is a
Ree  or Suzuki group. The results of the next subsection
allow to extend the definition of Sylow $\Phi$-subtori of
\cite{BM}  to any $(\bG,F)$ as in \ref{finite reductive}.

\subsection*{$F$-indecomposable tori}

\begin{definition} For $\bG, F$ as in \ref{finite reductive},
a non-trivial subtorus of $\bG$ is called {\em $F$-indecomposable} if it is
$F$-stable and contains no proper non-trivial $F$-stable subtorus.
\end{definition}

We say that a group $G$ is an almost direct product of subgroups $G_1$
and $G_2$ if they commute, generate $G$ and have finite intersection, and we
define similarly an almost direct product of $k$ subgroups by induction on $k$.
\begin{proposition}
For $\bG, F$ as in \ref{finite reductive},
any $F$-stable subtorus $\bT$ of $\bG$ is an almost direct product of
$F$-indecomposable tori $\bS_1,\ldots,\bS_k$
and $\card{\bT^F}=\card{\bS_1^F}\ldots\card{\bS_k^F}$.
\end{proposition}
\begin{proof}
An  $F$-stable subtorus $\bS$  corresponds to a  pure $F$-stable sublattice
$X'\subset  X:=X(\bT)$ (see for example \cite[III, Proposition 8.12]{Bl}). 
Let  $d$ be  the smallest  power of  $F$ which is a
split   Frobenius,   thus   on   $X(\bT)$   we  have  $F^{*d}=q^d\Id$.  Let
$\pi\in\End(X\otimes\BQ)$   be  a  projector  on  $X'\otimes\BQ$.  Then  in
$\End(X\otimes\BQ)$    we   can   define    the   $F$-invariant   projector
$\pi':=d\inv\sum_{i=1}^d  F^{*i}\pi  F^{*-i}$  and  $\Ker  \pi'\cap  X$  is
another $F$-stable pure sublattice which after tensoring by $\BQ$ becomes a
complement  to $X'\otimes\BQ$.  This corresponds  to an $F$-stable subtorus
$\bS'$  such that $K:=\bS\cap\bS'$ is  finite and $\bT=\bS\bS'$. Iterating,
we get the first part of the proposition.

The second part of the proposition results from the next two lemmas.
\end{proof}
\begin{lemma} \label{quotient by finite}
For $\bG, F$ as in \ref{finite reductive},
and $K$ an $F$-stable finite normal subgroup of $\bG$, then $\card{(\bG/K)^F}=
\card\GF$.
\end{lemma}
\begin{proof}
First, we notice that $K$ is central, thus abelian, since conjugating by
$\bG$ being continuous must be trivial on $K$.

Then,   the   Galois   cohomology   long   exact   sequence:  $1\to  K^F\to
\GF\to(\bG/K)^F\to   H^1(F,K)\to   1$   shows   the   result   using   that
$\card{K^F}=\card{H^1(F,K)}$.
\end{proof}
\begin{lemma} \label{almost direct}
Let $\bG$ as \ref{finite reductive} be an almost direct product of
$F$-stable connected subgroups $\bG=\bG_1\ldots\bG_k$. 
Then $\card\GF=\card{\bG_1^F}\ldots\card{\bG_k^F}$.
\end{lemma}
\begin{proof}
It  is enough to consider the case  $k=2$ and then iterate. Thus, we assume
$\bG=\bG_1\bG_2$  where $K=\bG_1\cap\bG_2$  is finite.  We quotient by $K$,
which makes the product direct, and apply Lemma
\ref{quotient by finite} twice.
\end{proof}

\begin{lemma} \label{F-indec}
Let  $\bS$ be an $F$-indecomposable torus, let $\eta$ be the smallest power
such  that  $q^\eta\in\BZ$,  and  let  $d$  be the smallest power such that
$F^{d\eta}$  is a  split Frobenius  on $\bS$.  Let $F^*=q\phi$ on $X(\bS)$;
then  the  characteristic  polynomial  $\Phi$  of  $\phi$  is  a  factor in
$\BZ[x,q\inv]$  of $\Phi_d(x^\eta)$,  where $\Phi_d(x)$  denotes the $d$-th
cyclotomic    polynomial.   Further   $q^{\deg\Phi}\Phi(x/q)\in\BZ[x]$   is
irreducible and $\card\SF=\Phi(q)$.
\end{lemma}
\begin{proof}
Since   $F^{*d\eta}$  acts  as  $q^{d\eta}$  on  $X:=X(\bS)$,  the  minimal
polynomial $P$ of $F^*$ divides $x^{d\eta}-q^{d\eta}$.

The polynomial $P$ is irreducible over $\BZ$, otherwise a proper nontrivial
factor $P_1$ defines an $F^*$-stable pure proper non-trivial sublattice
$\Ker(P_1(F^*))$ of $X$, which contradicts $F$-indecomposability of $\bS$.

It  follows that $X$ is a $\BZ[x]/P$-module by making $x$ act by $F^*$, and
$X\otimes\BQ[x]/P$  is a one-dimensional $\BQ[x]/P$-vector space, otherwise
a  proper nontrivial subspace would  define an $F^*$-stable pure sublattice
of  $X$. It follows that  $\dim\bS=\deg P=\dim X$ and  thus $P$ is also the
characteristic polynomial of $F^*$.

We have in $\BZ[x]$ the equality
$x^{d\eta}-q^{d\eta}=\prod_{d'|d}(q^{\eta\deg\Phi_{d'}}
\Phi_{d'}(x^\eta/q^\eta))$.  Since $P$ is irreducible it divides one of the
factors,  and since $d\eta$ is minimal such that $F^{*d\eta}=q^{d\eta}\Id$,
that  is minimal such that $P$  divides $x^{d\eta}-q^{d\eta}$, we have that
$P$    divides    $q^{\eta\deg\Phi_d}\Phi_d(x^\eta/q^\eta)$;   equivalently
$\Phi=q^{-\deg P}P(qx)$ divides $\Phi_d(x^\eta)$.

We have $\card\SF=\card{\Irr(\SF)}=\card{X/(F^*-1)X}=
\det(F^*-1)=(-1)^{\deg  P}P(1)=(-q)^{\deg\Phi}\Phi(1/q)$  where  the second
equality  reflects  the  well  known  group  isomorphism $\Irr(\bS^F)\simeq
X/(F^*-1)X$ and the third is a general property of lattices. Finally, since
$\Phi$  is  real  and  divides  $\Phi_d(x^\eta)$, its  roots are stable under
taking inverses,  thus $(-q)^{\deg\Phi}\Phi(1/q)=\Phi(q)$.
\end{proof}

We  call {\em  $q$-cyclotomic} the polynomials
$\Phi$  of Lemma \ref{F-indec}.  In other terms
\begin{definition}\label{qcyclo}
For  $q$ as in \ref{finite reductive}, where $q^\eta$ is the smallest power
of   $q$   in   $\BZ$,   we   call  {\em  $q$-cyclotomic}  the  monic
polynomials $\Phi\in\BZ[x,q\inv]$    such   that   
$q^{\deg    \Phi}\Phi(x/q)$   is   a
$\BZ[x]$-irreducible factor of some $x^{d\eta}-q^{d\eta}$.
\end{definition}
In   the  study   of  semisimple   reductive  groups   we  will  need  the
$q$-cyclotomic  polynomials of Lemma \ref{delta-cyclo}. Note that if $d$ is
minimal   in  Definition   \ref{qcyclo},  then   $\Phi$  is   a  factor  in
$\BZ[x,q\inv]$  of $\Phi_d(x^\eta)$. We  are interested in  that number $d$
rather  than $d\eta$, and to emphasize this we write $\Phi_{\eta,d}$ in the
following examples.
\begin{lemma}\label{delta-cyclo}
When $q\in\BZ$, the $q$-cyclotomic polynomials are the cyclotomic polynomials.

When  $q$  is  an  odd  power  of  $\sqrt 2$, the following polynomials are
$q$-cyclotomic:  $\Phi_{2,1}(x):=\Phi_1(x^2)$,  $\Phi_{2,2}(x):=\Phi_2(x^2)$, 
$\Phi_{2,6}(x):=\Phi_6(x^2)$,
the   factors   $\Phi'_{2,4}:=x^2+\sqrt 2x+1$   and   $\Phi''_{2,
4}:=x^2-\sqrt 2x+1$   of   $\Phi_4(x^2)$,   and   the  factors  $\Phi'_{2,
12}:=x^4+x^3\sqrt 2+x^2+x\sqrt2+1$ and $\Phi''_{2,
12}:=x^4-x^3\sqrt 2+x^2-x\sqrt2+1$ of $\Phi_{12}(x^2)$.

When  $q$  is  an  odd  power  of $\sqrt 3$, the following polynomials are
$q$-cyclotomic:  $\Phi_{2,1}(x)$,  $\Phi_{2,2}(x)$  and the factors
$\Phi'_{2,6}:=x^2+x\sqrt3+1$ and $\Phi''_{2,6}:=x^2-x\sqrt3+1$ of 
$\Phi_{6}(x^2)$.
\end{lemma}
\begin{proof}
When  $q\in\BZ$  the  formula  $P\mapsto  q^{-\deg  P} P(qx)$ establishes a
bijection   between   $\BZ[x]$-irreducible   factors   of   $x^d-q^d$   and
$\BZ[x]$-irreducible  factors of  $x^d-1$, that  is cyclotomic polynomials,
which gives the first case of the lemma.

For the other cases, we have to check for each given $\Phi$ that
$q^{\deg \Phi}\Phi(x/q)$ is in $\BZ[x]$ and irreducible.
\end{proof}

\begin{proposition}\label{q-indec}
Let $\bS, \eta, d, \Phi$ be as in \ref{F-indec} and let
$P=q^{\deg\Phi}\Phi(x^\eta/q^\eta)$  be  the  characteristic  polynomial of
$F^*$.  
\begin{enumerate}
\item Assume  that  either  $q\in\BZ$  or  that
$\BZ[x,q^{-\eta}]/P$ is integrally closed.  Then $\SF\simeq\BZ/\Phi(q)$.
\item  Let $m$ be a divisor of  $\Phi(q)$, and assume either that $d\in\{1,2\}$
and
$q\in\BZ$  or  that  $m$  is  prime  to  $d\eta$;  then  we  have a natural
isomorphism $\Irr(\SF)/m\Irr(\SF)\simeq \Ker(F^*-1\mid X(\bS)/m X(\bS))$.
\end{enumerate}
\end{proposition}
\begin{proof}
Proceeding  as  in  the  proof  of  Lemma  \ref{F-indec}  we  set $X=X(\bS)$ and
$\Xbar=X/(F^*-1)X\simeq\Irr(\SF)$. Letting $x$ act as $F^*$ makes $X$ into
a   $\BZ[x]/P$-module,   and   $\Xbar$   a  $\BZ[x]/(P,x-1)$-module.  Since
$\BZ[x]/(P,x-1)=\BZ/P(1)=\BZ/\Phi(q)$  we find that  the exponent of
$\Xbar$ divides $\Phi(q)$.

Let $A:=\BZ[x,q^{-\eta}]/P$. The extension $\BZ[x]/P\hookrightarrow A/P$ is
flat    thus    $\Xbar\otimes_{\BZ[x]/P}    A\simeq   X'/(F^*-1)X'$   where
$X'=X\otimes_{\BZ[x]/P}  A$;  and  since  the  exponent  of $\Xbar$ divides
$\Phi(q)$ which is prime to $q^\eta$, we have
$\Xbar\simeq\Xbar\otimes_{\BZ[x]/P}A$.  Under the assumptions of (1)
the ring $A$  is Dedekind: if
$\eta\ne 1$ then $A$  is integrally closed thus  Dedekind; 
if $\eta=1$ then $A\simeq\BZ[x,q\inv]/\Phi_d$ where
the  isomorphism is  given by  $x\mapsto x/q$,  and is  a localization  of the
Dedekind  ring $\BZ[x]/\Phi_d$ by $q$. Thus $X'$ identifies to a fractional
ideal $\fI$ of $A$ and $\Xbar\simeq\fI/(x-1)\fI$. If $e$ is the exponent of
$\Xbar$  we  have  thus  $e\fI\subset(x-1)\fI$,  which  implies  that $x-1$
divides  $e$  in  $A$.  This  in  turn  implies  that  the norm $(-1)^{\deg
P}P(1)=\Phi(q)$  of  $(x-1)$  divides  $e$  in  $\BZ$, thus $e=\Phi(q)$ and
$\Xbar\simeq\BZ/\Phi(q)$  and  the  same  isomorphism  holds  for  the dual
abelian group $\SF$.

For  (2), note that by construction $\Xbar/m \Xbar$ is the biggest quotient
of  $X$ on which both  $F^*-1$ and the multiplication  by $m$ vanish. It is
thus  equal to the biggest  quotient of $X/m X$  on which $F^*-1$ vanishes.
Thus the question is to see that $\Ker(F^*-1)$ has a complement in $X/m X$.

If  $q\in\BZ$ and $d\in\{1,2\}$ we have $P=x\pm q$ so $X\simeq\BZ$ on which
$F^*$ acts by $\mp q$ and $\Xbar=X/(q\pm1)$ of which $X/m X$ is a quotient,
so  $F^*-1$ vanishes on $X/m X$ which  is thus equal to $\Xbar/m \Xbar$ and
there is nothing to prove.

Assume  now $m$  prime to  $d\eta$. There  exists $R\in\BZ[x]$ such that in
$\BZ[x]$    we   have   $P=(x-1)R+P(1)$.   Taking   derivatives,   we   get
$P'=(x-1)R'+R$,  whence $R(1)=P'(1)$.  Let $\delta$  be the discriminant of
$P$;  we can  find polynomials  $M,N\in \BZ[x]$  such that $MP+NP'=\delta$,
which  evaluating at  $1$ gives  $M(1)P(1)+N(1)P'(1)=\delta$. Since  $q$ is
prime to $P(1)$, thus to $m$, and $\delta$ is a divisor of the discriminant
of  $X^{d\eta}-q^{d\eta}$,  equal  to  $q^{d\eta(d\eta-1)}(d\eta)^{d\eta}$,
thus prime to $m$, we find that $P'(1)$ is prime to $m$. In $(\BZ/m)[x]$ we
have  $P=(x-1)R$, thus  applied to  $F^*$ we  get that  on $X/m  X$ we have
$0=P(F^*)=(F^*-1)R(F^*)$,  whence  $\Ker(F^*-1)+\Ker(R(F^*))=X/m  X$. Since
$R(1)$  is prime to $m$, we can  write $1\equiv Q(x-1)+ aR$ in $(\BZ/m)[x]$
for  some $Q\in(\BZ/m)[x]$  and $a$  the inverse  $\pmod m$ of $R(1)$. This
proves that $\Ker(F^*-1)\cap\Ker(R(F^*))=0$ thus $X/ m X$ is the direct sum
of $\Ker(F^*-1)$ and $\Ker(R(F^*))$ q.e.d.
\end{proof}
\subsection*{Complex reflection cosets}
(1) to (3) below are classical results of Springer and Lehrer.
\begin{proposition}\label{coset}
Let $V$ be a finite dimensional vector space over a subfield $k$ of $\BC$, 
let $W\subset\GL(V)$ be a finite complex reflection group
and let $\phi\in N_{\GL(V)}(W)$, so that $W\phi$ is a reflection coset; let
$(d_1,\eps_1),\ldots,(d_n,\eps_n)$ be its generalized degrees (see for instance \cite[4.2]{B}).
For
$\zeta$ a root of unity define $a(\zeta)$ as the multiset of the $d_i$ such
that $\zeta^{d_i}=\eps_i$. Then:
\begin{enumerate}
\item  For any root of unity $\zeta$,  
the maximum dimension when $w\phi$ runs over $W\phi$ of
a $\zeta$-eigenspace of $w\phi$ on $V\otimes_k k[\zeta]$
is $\card{a(\zeta)}$.
\item 
For $w\phi\in W\phi$ denote $V_{w,\zeta}\subset V\otimes_k k[\zeta]$ its
$\zeta$-eigenspace. Assume $\dim V_{w,\zeta}=\card{a(\zeta)}$
and let $C=C_W(V_{w,\zeta})$ and 
$N=N_W(V_{w,\zeta})$. Then $N/C$ is a complex reflection group acting on
$V_{w,\zeta}$, with reflection degrees $a(\zeta)$.
\item  Any two subspaces $V_{w,\zeta}$ and $V_{w',\zeta}$ of dimension
$\card{a(\zeta)}$ are $W$-conjugate.
\item  For $w\phi$ as in (2) the natural actions of $w\phi$ on 
$N$ and $C$ induce the trivial action on $N/C$.
\item Let $a\in\BZ$ be such that $(W\phi)^a=W\phi$ and such that $\zeta$ and
$\zeta^a$ are conjugate by $\Gal(k[\zeta]/k)$.
Then for $w\phi$ as in (2) there exists $v\in N_W(N)\cap N_W(C)$ which 
conjugates $w\phi C$ to $(w\phi)^aC$.
\end{enumerate}.
\end{proposition}
\begin{proof}
For (1) see for instance \cite[5.2]{B}, for (2) see \cite[5.6(3) and (4)]{B}
and for (3) see \cite[5.6 (1)]{B}. (4) results from the observation that
if $n\in N$ and $v\in V_{w,\zeta}$ then
$(n\inv\cdot\lexp{w\phi}n)(v)=(n\inv w\phi n (w\phi)\inv)(v)=
(n\inv w\phi n)(\zeta\inv v)= (n\inv w\phi)(\zeta\inv n(v))=
(n\inv)(n(v))=v$ thus 
$n\inv\cdot\lexp{w\phi}n\in C$.

For (5),
$\Gal(k[\zeta]/k)$ acts naturally on $V\otimes_k k[\zeta]$, commuting with
$\GL(V)$, in particular with $W$ and $\phi$. If $\sigma\in\Gal(k[\zeta]/k)$
is such that $\sigma(\zeta)=\zeta^a$, let $\zeta^{a'}=\sigma\inv(\zeta)$.
Then $\sigma\inv(V_{w,\zeta})=V_{w,\zeta^{a'}}$.
It follows that $N=N_W(V_{w,\zeta^{a'}})$ and
$C=C_W(V_{w,\zeta^{a'}})$.

Now  since $a'$ is the inverse of $a$ modulo the order of $\zeta$ the space
$V_{w,\zeta^{a'}}$  is the $\zeta$-eigenspace of $(w\phi)^a$. By assumption
we  have  $(w\phi)^a\in  W\phi$.  Since  two maximal $\zeta$-eigenspaces of
elements  of  $W\phi$  are  conjugate  by  (3)  there exists $v\in W$ which
conjugates   $V_{w,\zeta}$  to  $V_{w,\zeta^{a'}}$,  and  $v\in  N_W(N)\cap
N_W(C)$  since $N=N_W(V_{w,\zeta^{a'}})$ and $C=C_W(V_{w,\zeta^{a'}})$. The
element  $v$  thus  conjugates  the  set  $w\phi  C$ of elements which have
$V_{w,\zeta}$  as $\zeta$-eigenspace to  the set $(w\phi)^a  C$ of elements
which have $V_{w,\zeta^{a'}}$ as $\zeta$-eigenspace.
\end{proof}

\subsection*{Generic Sylow subgroups}

We define the Sylow $\Phi$-subtori of $(\bG,F)$, first in the case when $\bG$
is quasi-simple, then in the case of descent of scalars.

From now on we assume $\bG$ semisimple. Then, if
$(d_1,\eps_1),\ldots,(d_n,\eps_n)$  are  the  generalized  degrees  of  the
reflection coset $W\phi$, we have (see \cite[11.16]{st})
\begin{equation}\label{order}
\card\GF=q^{\sum_i(d_i-1)}\prod_i (q^{d_i}-\eps_i).
\end{equation}

\begin{proposition}\label{sylows}
Let  $\bG$ be  as in  \ref{finite reductive}  and quasi-simple. Then we can
rewrite the order formula \ref{order} for $\card\GF$ as
\begin{equation}\label{orderbis}
\card\GF=q^{\sum_i(d_i-1)}\prod_{\Phi\in\cP}\Phi(q)^{n_\Phi}
\end{equation}
where  $\cP$  is  a  set  of  $q$-cyclotomic  polynomials,  and where $0\ne
n_\Phi=\card{a(\zeta)}$  (see \ref{coset}) for any root $\zeta$ of $\Phi$.
For  each  $\Phi\in\cP$  there  exists  a  non-trivial  $F$-stable subtorus
$\bS_\Phi$ of $\bG$ such that $\card{\bS_\Phi^F}=\Phi(q)^{n_\Phi}$.
\end{proposition}
We  note that if $\GF$  is a Ree or  Suzuki group, the $\eta$ of Definition
\ref{qcyclo} is $2$. Otherwise
$\eta=1$ and the $q$-cyclotomic polynomials are cyclotomic polynomials.

We  call any  $F$-stable torus  $\bS$ such  that $\card\SF$  is a  power of
$\Phi(q)$ a $\Phi$-torus, and tori $\bS_\Phi$ as above are called {\em
Sylow $\Phi$-subtori} of $(\bG,F)$ --- we abuse notation and call them
Sylow $\Phi$-subtori of $\bG$ when $F$ is clear from the context; 
they are the almost direct product of
$n_\phi$ $F$-indecomposable $\Phi$-tori.
\begin{proof}
Proposition \ref{sylows} is essentially in \cite{BM} but let us reprove it.

First,  we note  that assuming  $\card\GF$ has  a decomposition of the form
\ref{orderbis}, the value of $n_\Phi$ results from \ref{order}: let $\zeta$
be   any  root  of  $\Phi(x)$.  Then  $(x-\zeta)$  divides  $\Phi(x)$  with
multiplicity   one,  and  does  not   divide  any  another  $\Phi'(x)$  for
$\Phi'\in\cP$ since the $\Phi(x/q)$ are distinct irreducible polynomials in
$\BQ[x]$.  Thus $n_\Phi$  is the  number of  pairs $(d_i,\eps_i)$ such that
$x-\zeta$ divides $x^{d_i}-\eps_i$.

There is a decomposition of the form \ref{orderbis}:
if $\eta=1$ we get such a decomposition of $\card\GF$ 
by decomposing each term of \ref{order} into a product of cyclotomic
polynomials. Otherwise $\GF$ is a Ree  or Suzuki group, 
$\eta=2$ and $q$ is an odd power  of  $\sqrt  2$  or  $\sqrt  3$,  
and the  set $\cP$ and the decomposition of the form \ref{orderbis}
is given by what follows:
$$
\begin{array}{c|cc}
(\bG,F)&\card\GF&\text{generalized degrees of $W\phi$}\\
\hline
\lexp 2B_2(q^2)& q^4(\Phi_{2,1}\Phi'_{2, 4}\Phi''_{2, 4})(q)&
\{(2,1),(4,-1)\}\\
\lexp 2F_4(q^2)&q^{24}(\Phi_{2,1}^2\Phi_{2,2}^2
\Phi^{\prime2}_{2, 4}\Phi^{\prime\prime2}_{2,4}
\Phi_{2,6}\Phi'_{2, 12}\Phi''_{2, 12})(q)&
\{(2,1),(6,-1),(8,1),(12,-1)\}\\
\lexp 2G_2(q^2)& q^{6}(\Phi_{2,1}\Phi_{2,2}\Phi'_{2, 6}\Phi''_{2, 6})(q)&
\{(2,1),(6,-1)\}\\
\end{array}$$
Note  that for $\eta=2$ our ``$q$-cyclotomic  polynomials'' are the
``$(tp)$-cyclotomic  polynomials''  defined  in \cite[3.14]{BM}.

To construct the torus $\bS_\Phi$ for $\Phi\in\cP$, 
let us choose $\zeta$ a root of $\Phi$ and $w$ as in (2) of Proposition
\ref{coset}. Then if $\bT_w$ is a maximal torus of 
type $w$ with respect to $\bT$, so that $(\bT_w,F)\simeq(\bT,wF)$,
the characteristic polynomial of $w\phi$ on $X(\bT)$
has   $\Phi(x)^{n_\Phi}$  as   a  factor;  the   kernel  of
$\Phi(w\phi)$  on $X(\bT)$ is a pure  sublattice corresponding to a subtorus
$\bS_\Phi$ of $\bT_w$ such that $\card{\bS_\Phi^F}=\Phi(q)^{n_\Phi}$. 
\end{proof}

\begin{proposition}\label{sylow descent}
Let $(\bG,F)$ be as in \ref{finite reductive},
semisimple  and such  that the  Dynkin diagram  of $\bG$  has $n$ connected
components permuted transitively by $F$.
Then there exists a reductive group $\bG_1$ with isogeny $F_1$ such that
up to isomorphism  $\bG$ is  a  ``descent  of  scalars''
$\bG=\bG_1^n$  with $F(g_1,\ldots,g_n)=(g_2,\ldots,g_n,F_1(g_1))$.

Then $\GF\simeq\bG_1^{F_1}$, and if the scalar associated to $(\bG,F)$ is $q$
that associated to $(\bG_1,F_1)$ is $q_1:=q^n$.  
Thus we have
$\card\GF=q^{n\sum_i(d_i-1)}\prod_{\Phi\in\cP}\Phi(q^n)^{n_\Phi}$
where $d_i,\cP,n_\phi$ are as given by \ref{sylows} for $(\bG_1,F_1,q_1)$.

Here again, for $\Phi\in\cP$ there exists a Sylow $\Phi$-subtorus of $\bG$,
that is an $F$-stable subtorus
$\bS_\Phi$ such that $\card{\bS_\Phi^F}=\Phi(q^n)^{n_\Phi}$.
\end{proposition}
\begin{proof}
The proposition is obvious apart perhaps for the statement about the existence
of $\bS_\Phi$. This results in particular from the following lemma
that we need for future reference.

\begin{lemma}\label{descent}
In the situation of Proposition \ref{sylow descent},
let $(\bT,wF)$ where $\bT=\bT_1^n$ be a maximal torus of type 
$w=(1,\ldots,1,w_1)$ of $\bG$ and define $\phi$ on
$V=X(\bT)\otimes\BC$ (resp. $\phi_1$ on $V_1=X(\bT_1)\otimes\BC$) by
$F^*=q\phi$ (resp. $F_1^*=q_1\phi_1$).
Then if the characteristic polynomial of $w_1\phi_1$ is $P(x)$, 
that of $w\phi$ is $P(x^n)$.
Let $\Phi$ be a $q_1$-cyclotomic factor of $P$ (corresponding to a
$\BZ[x]$-irreducible factor of the characteristic polynomial of $w_1F_1^*$)
and let $\zeta$ be a root of $\Phi(x^n)$. Denote by $V_\zeta$ the
$\zeta$-eigenspace of $w\phi$ (resp. by $V_{1,\zeta^n}$ the
$\zeta^n$-eigenspace of $w_1\phi_1$). 

Let $\bS_1$ be the Sylow $\Phi$-subtorus of 
$(\bG_1,F_1)$ determined by $\Ker(\Phi(w_1\phi_1))$, and 
$\bS$ be the $wF$-stable subtorus of $\bT$ determined by 
$\Ker(\Phi((w\phi)^n))$.
Then $\bS$ is a Sylow $\Phi$-subtorus of 
$(\bG,F)$ and
$$N_W(V_\zeta)/C_W(V_\zeta)\simeq N_{W_1}(V_{1,\zeta^n})/C_{W_1}(V_{1,\zeta^n})
\simeq N_{\bG_1}(\bS_1)/C_{\bG_1}(\bS_1)
\simeq N_\bG(\bS)/C_\bG(\bS)$$
and we have an isomorphism $\bS^{wF}\simeq\bS_1^{w_1F_1}$ compatible with the 
actions
of $N_\bG(\bS)/C_\bG(\bS)$ and $N_{\bG_1}(\bS_1)/C_{\bG_1}(\bS_1)$
and the above isomorphism.
\end{lemma}
\begin{proof}
Let $X=X(\bT)$, $X_1=X(\bT_1)$. On $X\simeq X_1^n$ we have
$F^*(x_1,\ldots,x_n)=(x_2,\ldots,x_n,F_1^*(x_1))$, thus
$\phi(x_1,\ldots,x_n)=(q\inv x_2,\ldots,q\inv x_n,q_1q\inv x_1)$.
It follows by an easy computation that $V_\zeta$ is equal to the set of
$(x,(q\zeta) x,\ldots,(q\zeta)^{n-1} x)$ where
$x\in V_{1,\zeta^n}$, that
$C_W(V_\zeta)=\{(v_1,\ldots,v_n) \mid v_i\in C_{W_1}(V_{1,\zeta^n})\}$
and that $N_W(V_\zeta)=\{ (vv_1,\ldots,vv_n)\mid v\in
N_{W_1}(V_{1,\zeta^n}), v_i\in  C_{W_1}(V_{1,\zeta^n})\}$. This shows that  
$N_W(V_\zeta)/C_W(V_\zeta)\simeq N_{W_1}(V_{1,\zeta^n})/C_{W_1}(V_{1,\zeta^n})$.
Since when $\zeta$ runs over the roots of $\Phi(x^n)$ the $q_1\zeta^n$ are 
roots of the same $\BZ[x]$-irreducible
polynomial $q_1^{\deg\Phi}\Phi(x/q_1)$, the $\zeta^n$ are Galois conjugate thus
$C_{W_1}(V_{1,\zeta^n})$ (resp. $N_{W_1}(V_{1,\zeta^n})$) centralizes
(resp. normalizes) all the conjugate eigenspaces, whence our claim
that $N_{W_1}(V_{1,\zeta^n})/C_{W_1}(V_{1,\zeta^n})
\simeq N_{\bG_1}(\bS_1)/C_{\bG_1}(\bS_1)$.
Now $\Ker(\Phi((w\phi)^n))$ is the span of $V_\zeta$ for all roots $\zeta$
of $\Phi(x^n)$ and by the analysis above 
$C_W(V_\zeta)$ and $N_W(V_\zeta)$ are independent of $\zeta$, thus isomorphic
to $C_W(\bS)$ and $N_W(\bS)$.

We have the following commutative diagram
$$\begin{CD}
X@>{wF^*-1}>>X@>\Res>>\Irr(\bT^{wF})@>>>1\\
@VV\Sigma V @VV\Sigma V @VV\sim V\\
X_1@>{w_1F^*_1-1}>>X_1@>\Res>>\Irr(\bT_1^{w_1F_1})@>>>1\\
\end{CD}$$
where $\Sigma$ is the map $(x_1,\ldots,x_n)\mapsto x_1+\ldots+x_n$. Since we
have $\Sigma\circ (wF)^n=w_1F_1\circ\Sigma$, for any polynomial $Q$ the morphism
$\Sigma$ induces a surjective morphism $\Ker(Q((wF^*)^n))\to\Ker(Q(w_1F_1^*))$
whence for $Q=P$ a surjection $\Irr(\bS^{wF})\rightarrow\Irr(\bS_1^{w_1F_1})$; since
$\card{\bS^{wF}}$ is prime to $\card{\bT^{wF}/\bS^{wF}}$ this surjection must 
be an isomorphism.
Extended to $V=X\otimes\BC$, the map $\Sigma$ sends $V_\zeta$ to $V_{1,\zeta^n}$
and sends the action of $N_W(V_\zeta)/C_W(V_\zeta)$ to that
of $N_{W_1}(V_{1,\zeta^n})/C_{W_1}(V_{1,\zeta^n})$, whence the last
statement of the lemma.
\end{proof}
Note that any element of $W\phi$ is conjugate to an element of the form
$(1,\ldots,1,w_1)\phi_1$ so the form of $w$ in the statement of Lemma
\ref{descent} covers all the types of maximal tori.
\end{proof}

\begin{remark}
If  the generalized  degrees of  $W_1\phi_1$ are  $(d_i,\eps_i)_i$ those of
$W\phi$  are $(d_i,\eta_{i,j})$ where $\eta_{i,j}$ for $j\in\{1,\ldots,n\}$
runs  over the $n$-th  roots of $\eps_i$.  It follows that  $n_\Phi$ can be
defined   in  terms  of  $W\phi$ as it  is  also  the  number  of
$(d_i,\eta_{i,j})$ such that $\zeta^{d_i}=\eta_{i,j}$, where $\zeta$ is any
root of $\Phi(x^n)$.
\end{remark}
\begin{remark}
For $\Phi\in\cP(\bG)$, a  Sylow $\Phi$-subtorus of  $\bG$ is a ``power'' of a
subtorus   $\bS_0$  such  that   $\card{\bS_0^F}=\Phi(q)$.  If  $\bG$  is
quasi-simple, such a subtorus $\bS_0$ is $F$-indecomposable (since then the
polynomial  $\Phi$ is  $q$-cyclotomic). But  this is  no longer  true for a
descent of scalars. First, a cyclotomic polynomial in $x^n$ decomposes in
several cyclotomic polynomials according to the formula
$\Phi_d(x^n)=\prod_{\{\mu|n, \frac n\mu\text{ prime to } d\}}\Phi_{\mu d}(x)$
(see \cite[Appendice 2]{BM}).
But there could be further decompositions:
for instance, the characteristic polynomial of $F^*$ on
a  Coxeter torus of a semisimple group  $\bG$ of type $B_2$ over $\BF_2$ is
$x^2+4$,   which  is  $\BZ$-irreducible.  But   on  a  descent  of  scalars
$\bG\times\bG$, the characteristic polynomial of $F^*$ on a lift of scalars
of this torus is $x^4+4$ which is no longer $\BZ$-irreducible:
$x^4+4=(x^2+2x+2)(x^2-2x+2)$, so the torus seen inside the descent of scalars
is no longer $F$-indecomposable.

We  could  have  decomposed  $\card\GF$  into  a  product of $q$-cyclotomic
polynomials  corresponding to $F$-indecomposable  tori, but in  the case of
descent of scalars it was convenient to use larger tori.
\end{remark}
\begin{remark}
An arbitrary semisimple reductive group is of the form
$\bG=\bG_1\ldots\bG_k$,  an almost direct product of descents of scalars of
quasi-simple  groups $\bG_i$,  corresponding to  the orbits  of $F$  on the
connected components of the Dynkin diagram of $\bG$. Then we have
$\card\GF=\card{\bG_1^F}\ldots\card{\bG_k^F}$  by  Lemma \ref{almost
direct},  and  similarly, 
if  $\bS$  is  an $F$-stable torus of $\bG$, and
$\bS_i=\bS\cap\bG_i$,  then  $\card\SF=\card{\bS_1^F}\ldots\card{\bS_k^F}$.
This can be used to give a global decomposition of $\card\GF$, but the
polynomials $\cP$ in one factor could divide those in another.
For instance 
we could have $\Phi'_{2,4}$ for a factor of $\bG$ of type $\lexp 2B_2$ and
$\Phi_8$ for another factor of type $B_2$.
Because of this it is cumbersome to give a global statement.
\end{remark}

From  now on we  fix $(\bG,F)$ as  in \ref{sylow descent}, which determines
$q,n,$  and $\eta$ minimal  such that $q^{n\eta}\in\BZ$.  This allows in the
next  definition to  omit the  mention of  $\bG$ and  $F$ from the notation
$d(\ell)$.
\begin{definition}\label{d(l)}
Let  $\ell$  be  a  prime  number  different  from  $p$.  In the context of
\ref{sylow  descent}  we  define  $d(\ell)$  as  the  order  of  $q^{n\eta}
\pmod\ell$ ($\pmod 4$ if $\ell=2$).
\end{definition}
In particular $\ell|\Phi_{d(\ell)}(q^{n\eta})$.

The next proposition extends some of the Sylow theorems of \cite{BM}, and
introduces a complex reflection group $W_\Phi$ attached to each $\Phi$
in the set $\cP$ of \ref{sylows}.
\begin{proposition}\label{gensylow}
Under the assumptions of \ref{sylow descent}, 
let $\bT$ be an $F$-stable maximal torus of $\bG$ in an $F$-stable Borel
subgroup, and let $W\phi\subset\GL(X(\bT))$ be the reflection coset
associated to $(\bG,F)$.
Then for each $\Phi\in\cP$:
\begin{itemize}
\item[(1)] If $\zeta$ is a root of $\Phi(x^n)$ and $w$ is 
as in \ref{coset}(2), a maximal torus of $\bG$ of type $w$ with respect to
$\bT$ contains a unique Sylow $\Phi$-subtorus $\bS$.
\end{itemize}
For $\zeta,w$ as in (1) let $W_\Phi=N_W(V_\zeta)/C_W(V_\zeta)$
where $V_\zeta$ is the $\zeta$-eigenspace of $w\phi$ on $V=X(\bT)\otimes\BC$.
\begin{itemize}
\item[(2)] For $\bS$ as in (1) we have
$N_\GF(\bS)/C_\GF(\bS)=N_\bG(\bS)/C_\bG(\bS)\simeq W_\Phi$, and 
$W_\Phi$ can be identified to a subgroup of $\GL(X(\bS))$.
\item[(3)] The Sylow $\Phi$-tori of $\bG$ are $\GF$-conjugate.
\item[(4)] Let $\ell\ne p$ be a prime number, and assume
that  $\Phi$ divides $\Phi_{d(\ell)}$ (see Definition \ref{d(l)}).
Then unless $\ell=2$ and  $(\bG_1,F_1)$ is of type $\lexp 2G_2$,
any Sylow $\ell$-subgroup of $W_\Phi$ acts  faithfully on the subgroup
of $\ell$-elements $\SFl$ of $\SF$.
\end{itemize}
\end{proposition}
\begin{proof}
For  (1) we  consider a  torus $(\bT,wF)$  of type  $w$. Then a $wF$-stable
subtorus corresponds to the span of a subset of eigenspaces of $w\phi$ on $V$.
Since the
polynomials  $\Phi$ are prime to each other the polynomials $\Phi(x^n)$ are
also,  thus  $q\zeta$  is  root  of  no  other factor of the characteristic
polynomial  of $w\phi$  than $\Phi(x^n)$.  Thus the  $\bS$ defined in Lemma
\ref{descent}, which we will denote $\bS_0$, is unique.

Let  us  show  (2).  Let $(\bT_w,F,\bS)$ be conjugate to $(\bT,wF,\bS_0)$.
Let $\bL=C_\bG(\bS)$, which, as the centralizer of a torus, is
a Levi subgroup. Then we note that $N_\bG(\bS)\subset N_\bG(\bL)$. It follows
that we can find representatives of $N_\bG(\bS)$ modulo $\bL$ in $N_\bG(\bT_w)$
since for $n\in N_\bG(\bS)$ the torus $\lexp n\bT_w$ is another maximal torus
of $\bL$ which is thus $\bL$-conjugate to $\bT_w$. We thus get that
$N_\bG(\bS)/\bL=N_\bG(\bS,\bT_w)/(N_\bG(\bT_w)\cap \bL)$; 
transferring this to $\bT$ and then to $W$ we get 
$N_\bG(\bS,\bT_w)/(N_\bG(\bT_w)\cap \bL)\simeq N_W(\bS_0)/C_W(\bS_0)$ 
where $\bS_0$ is the subtorus of $\bT$
determined by $\Ker(P(wF^*))$ where $P=\Phi(x^n/q^n)$.
The action of $F$ is transferred to the action of $w\phi$ on this quotient.

That $N_W(\bS_0)=N_W(V_\zeta)$ and $C_W(\bS_0)=C_W(V_\zeta)$ was given in
\ref{descent}.

By \ref{coset}(4) we see that the action of $w\phi$ on 
$N_W(\bS_0)/C_W(\bS_0)$ is trivial, thus also that of $F$ on
$N_\bG(\bS)/C_\bG(\bS)$, thus
$N_\bG(\bS)/C_\bG(\bS)=(N_\bG(\bS)/C_\bG(\bS))^F
=N_\bG(\bS)^F/C_\bG(\bS)^F=N_\GF(\bS)/C_\GF(\bS)$, 
the second equality since
$\bL=C_\bG(\bS)$ is connected. Finally, the
last part of  (2) results from  the fact that  the representation of
$W_\Phi$  on $X(\bS_0)$, extended to $X(\bS_0)\otimes\BC$ has  as summand the 
representation of  $W_\Phi$  on  $V_\zeta$,  which  is the reflection 
representation, thus faithful.

(3) is a direct translation of \ref{coset}(3): when brought to subtori of
$\bT$   corresponding  to  eigenspaces  of  $w\phi$  (resp.  $w'\phi$)  the
$\GF$-conjugacy of two Sylow $\Phi$-subtori corresponds to the $W$-conjugacy 
of the corresponding eigenspaces.

For (4) we first remark that we can reduce to the case where $\bG$ is
quasi-simple, using \ref{descent}.
Thus either $q\in\BZ$ or $\GF$ is a Ree or a Suzuki group. Let $\delta$ be
the order of the coset $W\phi$, that is the smallest integer such that
$(W\phi)^\delta=W$. We have $\delta\in\{1,2,3\}$.
We first show the
\begin{lemma}\label{faithful}
If $\bG$ is quasi-simple and we are in one of the cases:
\begin{enumerate}
\item $q\in\BZ$ and $\delta\in\{1,2\}$.
\item $q\in\BZ$, $\delta=3$ and $d$ is prime to 3.
\item $q$ is an odd power of $\sqrt 2$ and $\ell=3$.
\end{enumerate}
then $W_\Phi$ acts  faithfully on $\SFl$.
\end{lemma}
\begin{proof}
On  $X(\bT)\otimes\BQ(q\inv)$  we  have  $wF^*=qw\phi$.  The characteristic
polynomial $Q$ of $wF^*$ on $X(\bS)$ is
$q^{n_\Phi\deg\Phi}\Phi(x/q)^{n_\Phi}$;   as  $wF^*$   is  semisimple,  the
minimal  polynomial of $wF^*$ is $P=q^{\deg\Phi}\Phi(x/q)$. We can identify
$X(\bS)$ with $\Ker(P(qw\phi))$ on $X(\bT)$. As in the proof of Proposition
\ref{q-indec},   if  $X=X(\bS)$  we  can  make  $X'=X\otimes\BZ[q\inv]$  an
$A$-module where $A=\BZ[x,q^{-\eta}]/P$. Under the assumptions of the lemma
$A$  is a Dedekind ring. This results  from the proof of \ref{q-indec}(1) when
$q\in\BZ$. In the remaining case (3) of Lemma \ref{faithful}, 
$\eta=2$ and the order of $q^2\pmod 3$ is 
$2$, thus $\Phi=x^2+1$ and $P=x^2+q^2$; we have
$A=\BZ[x,q^{-2}]/P\simeq\BZ[1/2,\sqrt{-2}]$   which  is  integrally  closed
(thus Dedekind)
since  localized  of  $\BZ[\sqrt{-2}]$  which  is  integrally closed. As an
$A$-module  of rank $n_\Phi$, the module $X'$ is a sum of projective rank 1
submodules thus $\bS$ is a product of $n_\Phi$ copies of a
$wF$-indecomposable  torus.  By Proposition \ref{gensylow}(2) we  
can  identify  $W_\Phi$  to  a subgroup of
$\GL(X)$. With the notations of \ref{q-indec}, since the assumption
of \ref{q-indec}(1) is satisfied,
$\Xbar:=X/(wF^*-1)X\simeq\Irr(\bS^{wF})$  is isomorphic  to 
$(\BZ/\Phi(q))^{n_\Phi}$.
The  representation of  $W_\Phi$ on  $X$ reduces to $\Xbar$. We will show
it is faithful on $\Xbar/\ell\Xbar$ (or $\Xbar/4\Xbar$ when $\ell=2$).

If  $q\in\BZ$ and $\ell=2$ then $d\in\{1,2\}$  and we can apply Proposition
\ref{q-indec}(2)  taking $m=4$.  We get  that $\Xbar/4 \Xbar\simeq \Ker
(wF^*-1\mid  X/4  X)$.  We  have  as  observed  in the proof of Proposition
\ref{q-indec}  that $\Ker(wF^*-1)=X/4 X$ and the representation of $W_\Phi$
on  $\Xbar/4 \Xbar$,  which  is  a quotient of $\Irr(\bS_\ell^{wF})$,
is faithful by Lemma \ref{reduction not id}.

If  $q\in\BZ$ and $\ell\ne 2$ then  $d$ is prime  to $\ell$;
and in case (3) of Lemma \ref{faithful}
$\eta=2$, $\ell=3$ thus $d=2$ and $\ell$ is prime to $d\eta$.
In both cases we can apply Proposition
\ref{q-indec}(2)  with  $m=\ell$ to get that $\Xbar/\ell\Xbar
\simeq\Ker(wF^*-1\mid X/\ell X)$. We know by Lemma \ref{reduction not id}
that  the representation of $W_\Phi$ on $X/\ell X$ is faithful and we would
like  to conclude that  it is faithful  on the submodule $\Ker(wF^*-1)$. We
use  the element $v$ given by  Proposition \ref{coset}(5): it preserves the
kernel  of $\Phi(w\phi)$ thus induces an  element of $\GL(X)$ which defines
an automorphism $\sigma$ of $W_\Phi$ which sends $w\phi$ to $(w\phi)^a$, so
it  remains true after reduction $\pmod\ell$ that $\sigma$ sends $w\phi$ to
$(w\phi)^a$,  thus permutes the eigenspaces of  $wF^*$ on $X/\ell X$: since
$d$  is the order of $q\pmod \ell$, all the primitive $d$-th roots of unity
live  in $\BF_\ell$ and  the eigenvalues of  $wF^*$ are the  product of one
primitive $d$-th root of unity, which is $q$, by the other primitive $d$-th
roots   of  unity  so  are  of  the   form  $q^{1-a}$  where  a  runs  over
$(\BZ/d)^\times$.   And   under   the   assumption   $(W\phi)^a=W\phi$   of
\ref{coset}(5)   we   can   find   $v$   thus   $\sigma$  which  sends  the
$q^{1-a}$-eigenspace of $wF^*$ to the $q^{1-1}=1$-eigenspace.

If  every $a$  prime to  $d$ has  a representative  in $1+\delta\BZ$ we can
satisfy  $(W\phi)^a=W\phi$  for  such  $a$  thus  every  eigenspace is
isomorphic  as  a  $W_\Phi$-module  to  $\Ker(wF^*-1)$.  Then  $W_\Phi$  is
faithful  on  the  whole  $X/\ell  X$  if  and  only  if  it is faithful on
$\Ker(wF^*-1)$,  thus we conclude.  If $a\equiv1\pmod{\gcd(d,\delta)}$ then
by   Bezout's  theorem  there  exist   integers  $\alpha,\beta$  such  that
$a=1+\alpha  d+\beta\delta$,  and  then  $a-\alpha  d\in  1+\delta\BZ$ is a
representative of $a$.

If $\delta=1$ or $\delta=2$
then every $a$ prime to $d$ is $\equiv 1 \pmod{\gcd(d,\delta)}$ and we
conclude. We conclude similarly if $\delta=3$ and $d$ is prime to $3$,
or in case (3) of Lemma \ref{faithful} since in this case $d=2$.
\end{proof}
When  $q\in\BZ$ the only case not covered  by the lemma is $\lexp 3D_4$ and
$d$  divisible by $3$, that is $d\in\{3,6,12\}$. But in this case $\ell>3$,
since  $d$  is  the  order  of  $q\pmod\ell$,  thus  $\card W$ is prime to
$\ell$ and a fortiori the Sylow $\ell$-subgroup of $W_\Phi$ is trivial.

For  the Ree and Suzuki groups we do not have to consider $\lexp2B_2$ since
$W$  is a $2$-group and $\ell\ne p$, and the groups $\lexp 2G_2$ since only
the  prime $\ell=2$ divides $\card  W$ and is different  from $p$, and this
case is excluded in the proposition.

For the groups $\lexp 2F_4$ the only prime $\ell\neq p$
such that $\ell|\card W$ is $\ell=3$ and we are in case (3) of the lemma.
\end{proof}
The  Ree group $\lexp 2G_2$ with $\ell=2$ is a genuine counterexample since
the Sylow $2$-subgroups of $\lexp2G_2(q)$ are isomorphic to $(\BZ/2)^3$.
\section{The structure of the Sylow $\ell$-subgroups}
\begin{definition}\label{D(l)}
Let $\bG, F, \bG_1,\cP$ and $n$ be as in \ref{sylow descent} and let 
$\ell\ne p$ be a prime number.
We define  $D(\ell)$ as the  set of integers $d$ such that for some
$\Phi\in\cP$ dividing $\Phi_d(x^\eta)$  we  have  $\ell|\Phi(q^n)$, where
$\eta$ is as in Definition \ref{d(l)}.
\end{definition}
The following proposition is \cite[Th\'eor\`eme 1]{En}
when $\eta=1$; we give here a shorter proof.
Since \cite{En} was written, Malle (\cite[5.14 and 5.19]{malle})
has published a proof of (2) below --- thus implicitly of (1) also---
when $\eta=1$ (giving more, see Theorem \ref{unicity}).
\begin{theorem}\label{Enguehard}
Assume in the situation of \ref{D(l)} that $D(\ell)\ne\emptyset$, or
equivalently that $\ell|\card\GF$. Then
\begin{enumerate}
\item $d(\ell)\in D(\ell)$.
\item There exists a unique $\Phi\in\cP$ such that
 $\ell|\Phi(q^n)$ and $\Phi$ divides $\Phi_{d(\ell)}(x^\eta)$.
If $\bS$ is a Sylow $\Phi$-torus then $N_\bG(\bS)$  contains a
Sylow $\ell$-subgroup of $\GF$ which is an extension of
$(Z^0C_\bG(\bS))^F_\ell$  by a Sylow $\ell$-subgroup of $W_\Phi$.
\item The Sylow $\ell$-subgroups of $\GF$ are abelian if and only if 
 $\card{D(\ell)}=1$ (which is equivalent to $W_\Phi$ being an $\ell'$-group), 
 apart  from the exception where $(\bG_1,F_1)$  is of type $\lexp 2G_2$ and
 $\ell=2$  in which  case $\card{D(\ell)}=2$  and $\card{W_\Phi}=6$ but the
 $2$-Sylow is abelian, isomorphic to $(\BZ/2)^3$.
\end{enumerate}

Further,  if $\bS$ is as in (2),
then  $(Z^0C_\bG(\bS))^F_\ell=\SFl$ except if:
\begin{itemize}
\item $\ell=3$ and $\bG_1$ of type $\lexp 3D_4$. 
\item $\ell=2, d=1$ and for some odd degree $\eps_i=-1$. Equivalently $\bG_1$
is non-split and has an odd reflection degree, that is, is one of $\lexp 2A_n$, 
$\lexp 2D_{2n+1}$ or $\lexp 2E_6$.
\item $\ell=2, d=2$ and for some odd degree $\eps_i=1$; equivalently $\bG_1$
is split and has an odd reflection degree, that is, is
one of $A_n (n>1)$, $D_{2n+1}$ or $E_6$.
\end{itemize}
In the above exceptions, $Z^0C_\bG(\bS)=C_\bG(\bS)$ is a maximal torus
of $\bG$.
\end{theorem}
\begin{proof}
Let us note that to prove (2) when we are not in an exception, that is the
stronger statement that a Sylow $\ell$-subgroup is in an extension of
$\SF$ by a Sylow $\ell$-subgroup of $W_\Phi$,  it is  enough to prove that
$$\vl(\card\GF)=\vl(\card\SF)+\vl(\card{W_\Phi})\eqno(*)$$
where $\vl$ denotes the $\ell$-adic valuation,
and in the exceptions, if we have proved that 
$Z^0C_\bG(\bS)=C_\bG(\bS)$ it is enough to show
$$\vl(\card\GF)=\vl(\card{C_\bG(\bS)^F})+\vl(\card{W_\Phi})\eqno(**)$$

Note  also that by the definition of $d(\ell)$ and $D(\ell)$ in Proposition
\ref{sylow  descent}, assertion  (1) as  well as  formulae (*) and (**) are
equivalent  in $\bG$ and $\bG_1$, that  is we may assume $\bG$ quasi-simple
to  prove  them  which  we  do  now.  Also,  in view of (2) and Proposition
\ref{gensylow}(4),   (3)  reduces  to   proving:

(3') $\card{D(\ell)}=1$  is equivalent to $W_\Phi$ being an $\ell'$-group.

We first look at the case of a Ree or Suzuki group, where $\eta=2$.

Let  us  prove  (1)  first.  By  Lemma  \ref{divcyclo}  if  $\ell$  divides
$\card\GF$   then  there   is  an   element  of   $D(\ell)$  of   the  form
$d(\ell)\ell^b$  with $b\ge 0$. By inspecting the order formula for $|\GF|$
given in the proof of \ref{sylows} the elements of $D(\ell)$ have all their
prime   factors  in  $\{2,3\}$,  so  $b>0$  implies  $\ell\in\{2,3\}$  thus
$d(\ell)\in\{1,2\}$;  inspecting  again  the  formula,  we  see  that  then
$d(\ell)$ in $D(\ell)$ and that $\card{D(\ell)}=1$ unless $\ell\in\{2,3\}$.

To  prove  (2)  for  $\ell\notin\{2,3\}$,  we  observe  there  is  a single
$\Phi\in\cP$ such that $\ell|\Phi(q)$ since the two numbers
$\Phi'_{2,4}(q),\Phi''_{2,  4}(q)$ are  prime to  each other,  and the same
observation  applies to  $\Phi'_{2, 6}(q),\Phi''_{2,  6}(q)$ and $\Phi'_{2,
12}(q),  \Phi''_{2, 12}(q)$. Thus for  $\ell\notin\{2,3\}$ assertions (3')
and (*) are obvious since $\card\GF_\ell=\card\SF_\ell$ and $\ell\not| \card W$.

Let  us prove (*) for $\ell\in\{2,3\}$;  since $\ell\ne p$ and the elements
of $D(\ell)$ have only $2$ as prime factor in the case $\lexp 2B_2$, we have
just to consider:

\begin{itemize}
\item  $\ell=3$ for $\lexp 2F_4$: we have $d(3)=2$, $W_{\Phi_{2,2}}=G_{12}$
of  order $48$;  the only  factor $\Phi(q)$  with a  value divisible by $3$
apart    from    $\card\SF=\Phi_{2,2}(q)^2$    is    $\Phi_{2,6}(q)$    and
$v_3(\Phi_{2,6}(q))=1=v_3(\card{G_{12}})$ which proves this case.
\item $\ell=2$ for $\lexp 2G_2$: we have $d(2)=2$ and
$\card{W_{\Phi_{2,2}}}=6$;  the only  factor $\Phi(q)$  with an  even value
apart from $\card\SF=\Phi_{2,2}(q)$ is $\Phi_{2,1}(q)$ and
$v_2(\Phi_{2,1}(q))=1=v_2(\card{W_\Phi})$ which proves this case.
\end{itemize}
We have seen (3') along the way.

Now  we  look  at  the  other  quasi-simple groups thus $\eta=1$. We notice
generally  that,  assuming  we  have  proved (1) then if $\card{D(\ell)}=1$
assertion  (2) is trivial  since a Sylow  $\ell$-subgroup is then in $\bS$,
and (3') reduces to checking that $W_\Phi$ is an $\ell'$-group.

We consider separately $\lexp 3D_4$ where $\card{\lexp
3D_4(q)}=q^{12}(\Phi_1^2\Phi_2^2\Phi_3^2\Phi_6^2\Phi_{12})(q)$.      Again,
since the only prime factors of elements of $D(\ell)$ are $\{2,3\}$, we see
that  $d(\ell)\in D(\ell)$  except possibly  if $\ell\in\{2,3\}$;  but in
that  case $d(\ell)\in\{1,2\}$  and there  is a factor $\Phi_{d(\ell)}(q)$,
whence  (1).  Since  $\card  W=3\cdot  2^6$  assertion  (3') is proved when
$D(\ell)=1$.  It remains to prove (2)  when $\ell\in\{2,3\}$. In both cases
$W_{\Phi_{d(\ell)}}=W(G_2)$ and by Lemma \ref{divcyclo}
$\vl(\card\GF/\card\SF)=2$.  If $\ell=2$  then $2=\vl(\card{W(G_2)})$ which
proves  (*).  If  $\ell=3$  a  Sylow  $\Phi$-subtorus  $\bS$  is in a torus
$\bT_w=C_\bG(\bS)$  where $w=1$ if $d=1$ (resp.  $w=w_0$ if $d=2$). We have
$\card{\bT_1^F}=\Phi_1(q)^2\Phi_3(q)$ (resp.
$\card{\bT_{w_0}^F}=\Phi_2(q)^2\Phi_6(q)$)  which has same $3$-valuation as
$\card\GF/\card{W_\Phi}$ which proves (**).

In the remaining cases $\eps_i=\pm  1$  for  all  $i$.  Let us set 
$\zeta_d=e^{2i\pi/d}$. We have $\Phi=\Phi_{d(\ell)}$ and 
$\vl(\card\SF)=\card{a(\zeta_{d(\ell))})}\vl(\Phi_{d(\ell)}(q))$.

We  first  treat  the  case  $\ell$  odd.  We  have  $a(\zeta_d)=\{ d_i\mid
\zeta_d^{d_i}=\eps_i\}$ and $\card{W_\Phi}= \prod_{d_i\in
a(\zeta_{d(\ell)})}d_i$.  By Lemma \ref{divcyclo},  a factor $\Phi_e(q)$ of
$\card\GF$  can contribute  to the  $\ell$-valuation only  if $e$ is of the
form  $d(\ell)\ell^b$ for some $b\ge 0$.  Further  such  a  factor  appears  if  and  only if
$a(\zeta_e)\ne\emptyset$, that is for some $i$ we have
$\zeta_{d(\ell)\ell^b}^{d_i}=\eps_i$.  Since  $\ell$  is  odd  raising this
equality  to the  power $\ell^b$  gives $\zeta_{d(\ell)}^{d_i}=\eps_i$ thus
$d_i\in  a(\zeta_{d(\ell)})$  and  in  particular $d(\ell)\in D(\ell)$. And
$\zeta_{d(\ell)\ell^b}^{d_i}=\eps_i$  implies that  $\ell^b$ divides $d_i$.
Thus  only the $d_i$ in  $a(\zeta_{d(\ell)})$ contribute to $\vl(\card\GF)$
and each of them contributes
$\vl(\Phi_{d(\ell)}(q))+\vl(\Phi_{d(\ell)\ell}(q))+
\ldots+\vl(\Phi_{d(\ell)\ell^{\vl(d_i)}}(q))$. By Lemma \ref{divcyclo} this
is $\vl(\Phi_{d(\ell)}(q))+\vl(d_i)$. Summing over $d_i\in a(\zeta_{d(\ell)})$ 
proves (*).

It  remains the case $\ell=2$ where  we proceed similarly. We have
$d(2)\in\{1,2\}$. If $d(2)=1$  then $a(1)=\{d_i\mid \eps_i=1 \}$.
Thus  the  condition  $\zeta_{2^b}^{d_i}=\eps_i$  is  still  equivalent  to
$2^b|d_i$;  but there could  be some more  solutions of this equation than
elements of $a(1)$ when
$b=1$:  any odd  $d_i$ such  that $\eps_i=-1$  brings an additional factor
$1=v_2(\Phi_2(q))$.  If  $d(2)=2$  then  $a(-1)=\{d_i\mid \eps_i=(-1)^{d_i}\}$.
The contribution of the
even  $d_i$ can be worked out as before;  but this time the odd $d_i$ where
$\eps_i=1$  bring additional  factors $v_2(\Phi_1(q))$.  
In  the exceptions in each case $C_\bG(\bS)$ is a maximal torus of type $1$
or  $w_0$; looking at the  orders of these tori,  they contain enough extra
$\Phi_1$  or $\Phi_2$ factors  (which correspond to  the eigenvalues $1$ or
$-1$ of $\phi$ or $w_0\phi$) to compensate the discrepancy.

Let   us  show  now  (3'),  which reduces to proving
that   $\card{D(\ell)}>1$ implies  $\vl(\card{W_\Phi})>0$. Thus we assume
$\card{D(\ell)}>1$. We first do the case $\ell=2$; then
$d(\ell)\in\{1,2\}$ from which it follows, since the $1$ and $-1$-eigenspaces
are defined over the reals, that $W_\Phi$ is a Coxeter group,
whose  order is always  even. We consider finally $\ell$ odd; 
then $D(\ell)\owns d(\ell)$  and $d(\ell)\ell^a$ for  some $a>0$. 
But  we have seen above that
there  exists a factor $\Phi_{d(\ell)\ell^a}(q)$  only if $\ell^a| d_i$ for
some $d_i$ in $a(\zeta_{d(\ell)})$.
\end{proof}

We   remark  that  if   $\ell$  divides  only   one  $\Phi_d(q)$,  a  Sylow
$\ell$-subgroup  $S$  lies  in  a  single  Sylow  $\Phi$-torus  $\bS$  (the
intersection  of two  tori has  lower dimension  so cannot  have same order
polynomial). It follows that $N_\GF(S)=N_\GF(\bS)$ and
$C_\GF(S)=C_\GF(\bS)$. This observation is a start for describing the
$\ell$-Frobenius category of $\GF$ in terms of the category of
$\zeta_d$-eigenspaces of $W_{\Phi_d}$.

In general, one can deduce the following unicity theorem from the work of 
Cabanes, Enguehard and Malle.

\begin{theorem}\label{unicity} Consider  $\bG,F,n,\bG_1,q$  as in
\ref{sylow descent} with
$q^n\in\BZ$ and let $\Phi$ as defined in Theorem \ref{Enguehard}, (2). 
Assume that we are not in one of the following cases:
\begin{itemize}
\item $\ell=3$, $\bG_1$ simply connected of type $A_2$, $\lexp 2A_2$ or $G_2$.
\item $\ell=2$, $\bG_1$ simply connected of type $C_n, n\ge 1$.
\end{itemize}
Let  $Q$ be a Sylow $\ell$-subgroup of  $\GF$. There is a unique Sylow
$\Phi$-subtorus $\bS$ of $\bG$ such that $Q\subseteq N_\bG(\bS)$.
\end{theorem}
\begin{proof}  In the context  of Theorem \ref{Enguehard}(2),  let $Q$ be a
Sylow $\ell$-subgroup of $\GF$ contained in $N_\bG(\bS)$; then according to
\cite{cabanes},  $\bS^F_\ell$ is often characteristic  in $Q$ (for example
when $l\ge 5$), thus in these
cases   $N_\GF(Q)\subseteq   N_\bG(\bS^F_\ell)$.   Using  inductively  that
property and inspecting small cases, G.~Malle has proved the inclusion
\begin{equation}\label{normalizer}
N_\GF(Q)\subseteq N_\bG(\bS) 
\end{equation}
for  all quasi-simple groups  $\bG$ short of  the cases excluded in Theorem
\ref{unicity},  see \cite[Theorems 5.14  and 5.19]{malle}. Here  $\bS$ is a
Sylow  $\Phi_{d(\ell)}$-subtorus  of  $(\bG,F)$  as  defined  in Definition
\ref{d(l)}  with  $\eta=1$  (note  that  $N_{\GF}(Q)\subseteq N_{\bG}(\bS)$
implies $Q \subseteq N_{\bG}(\bS)$).

We  first verify that the last inclusion holds more generally in a "descent
of  scalars".  With  hypotheses  and  notations  of  Proposition \ref{sylow
descent}  and Lemma \ref{descent} assume $q^n\in\BZ$. If $e=d(\ell)$ is the
order   of   $q^n$   modulo   $\ell$,  take  $\Phi=\Phi_e\in\cP$,  defining
$\bS=\bS_\Phi$  and $\bS_1$. There  is a morphism  from $\bG$ onto $\bG_1$,
sending  $\bS$ to  $\bS_1$, with  restriction an  isomorphism from $\GF$ to
$\bG^F_1$.  Then a  Sylow $\ell$-subgroup  $Q_1$ of  $\bG^F_1$ contained in
$N_{\bG_1}(\bS_1)$  is the isomorphic image  of a Sylow $\ell$-subgroup $Q$
of  $\GF$  contained  in  $N_{\bG}(\bS)$.  The  inclusion  \ref{normalizer}
written  with  $(\bG_1,F_1,Q_1,\bS_1)$  instead  of $(\bG,F,Q,\bS)$ implies
\ref{normalizer} in $(\bG,F)$.

From \ref{normalizer} the  unicity  of  $\bS$,  given   $Q$,  follows:
\begin{lemma}\label{bibi}   Let   $\Phi\in\cP$,   let   $\bS$  be  a  Sylow
$\Phi$-subtorus  of $(\bG,F)$ and $Q$ a  Sylow $\ell$-subgroup of $\GF$. If
$N_\GF(Q)\subseteq N_\bG(\bS)$, then $\bS$ is the unique Sylow $\Phi$-torus
of $(\bG,F)$ such that $Q\subseteq N_\bG(\bS)$.
\end{lemma}
\begin{proof}
Assume  $Q\subseteq  N_\bG(\bS')$  for  some  Sylow  $\Phi$-torus $\bS'$ of
$(\bG,F)$.  By Proposition \ref{gensylow} there  exists $g\in\GF$ such that
$\bS=(\bS')^g$,  hence  $Q^g\subseteq  N_\bG(\bS)$.  By  Sylow's theorem in
$N_\bG(\bS)^F$,  $Q=Q^{gh}$  for  some  $h\in  N_\bG(\bS)^F$  hence  $gh\in
N_\bG(\bS)$ by our hypothesis.
\end{proof}
\end{proof}
	
\section{Appendix}
We gather here arithmetical lemmas used above.
\begin{lemma}\label{div}
Let $x,f,\ell\in\BN$ where $\ell$ is prime, and assume
$x\equiv 1\pmod\ell$ (resp. $\pmod 4$ if $\ell=2$).  Then 
$\vl(\dfrac{x^f-1}{x-1})=\vl(f)$.  
\end{lemma}
\begin{proof}
From $\dfrac {x^{f_1f_2}-1}{x-1}=\dfrac {x^{f_1f_2}-1}{x^{f_2}-1}
\dfrac{x^{f_2}-1}{x-1}$ we see that it is enough to show the lemma when $f$ is
prime. 
We have $\dfrac{x^f-1}{x-1}=f+\sum_{i=2}^{i=f}  (x-1)^{i-1}{f\choose  i}$.
Let $S$ be this last sum; we have $S\equiv f\pmod \ell$,
since  $x-1\equiv 0\pmod \ell$, thus $S$ is prime to $\ell$ when $f\ne\ell$
which shows the lemma in this case. When
$f=\ell$  then all the terms of $S$ but the first one and possibly 
the last one are divisible by
$\ell^2$  since $\ell\choose i$ is divisible  by $\ell$ when $2\le i<\ell$;
the last term is divisible by $\ell^2$ when $\ell-1\ge 2$
which fails only for $f=\ell=2$;  but when $\ell=2$
we  have arranged that $\vl(x-1)\ge 2$  and this time $2(f-1)\ge 1$; thus
$S\equiv f\pmod \ell^2$, whence the lemma.
\end{proof}
The following lemma is in \cite[5.2]{malle};
a short elementary proof results immediately from Lemma \ref{div}.
\begin{lemma}\label{divcyclo}
Let $q,\ell\in\BN$ where $\ell$ is prime.
Let $d$ be the order of $q\pmod \ell$ (or $\pmod 4$ if $\ell=2$).
Then $\ell$ divides $\Phi_e(q)$ if and only if $e$ is
of the form $d\ell^b$ with $b\in\BN$ (or additionally $b=-1$ when $\ell=d=2$), 
and $\vl(\Phi_{d\ell^b}(q))=1$ if $b\ne 0$.
\end{lemma}
The following lemma is in \cite{M}; we give the proof since it is very short
and the original German proof may be less accessible.
\begin{lemma}\label{reduction not id}
Let $m\in\BN, m>2$. Then the kernel of the reduction map
$\GL(\BZ^n)\to\GL((\BZ/m)^n)$ is torsion-free.
\end{lemma}
Note that the bound $m>2$ is sharp since $-\Id\equiv\Id\pmod 2$.
\begin{proof}
Let $w\in\GL(\BZ^n)$ be of finite order, $w\ne\Id$
and assume its reduction $v=\Id$. We will derive a contradiction.

Possibly replacing $w$ by a power, we may assume that $w$ is of prime order
$p$. 

Also  $\GL(\BZ^n/m)=\prod_i\GL(\BZ^n/p_i)$ where
$m=\prod_i  p_i$ is the decomposition of $m$ into prime powers, thus we
may assume that $m$ is a prime power.

Since  $w$ is of order  $p$, the polynomial $\Phi_p(x)$  is a factor of the
characteristic  polynomial of $w$. The  characteristic polynomial of $v$ is
the reduction $\pmod m$ of that of $w$, thus we must have
$\Phi_p(x)\pmod m\equiv(x-1)^{p-1}$;  in particular ${p-1\choose 1}\equiv
-1 \pmod m$ thus $m|p$ which implies $m=p$.

Write now $w=\Id+x m^a$ where $x\pmod m\not\equiv0$ and $a\in\BN$.
Then the equation $w^m=\Id$ gives
$\sum_{i=1}^m {i\choose m} x^im^{ai}=0$, which after dividing
by $m^{a+1}$ becomes
$x=-\sum_{i=2}^m {i\choose m} x^im^{a(i-1)-1}$ where all coefficients
on the right-hand side are divisible by $m$ (since $m\ge 3$), which
contradicts $x\pmod m\not\equiv0$.
\end{proof}


\begin{thebibliography}{rt}
\bibitem[Borel]{Bl} A.~Borel, ``Linear algebraic groups'', Springer GTM no.
126, 2nd ed. 1991.
\bibitem[Brou\'e]{B}  M.~Brou\'e,  ``Introduction  to  complex reflection
groups  and their  braid groups'', {\sl Lecture  Notes in  Mathematics \bf
1988}, Springer-Verlag, Berlin, 2010.
\bibitem[Brou\'e-Malle]{BM} M.~Brou\'e and G.~Malle, ``Th\'eor\`emes de Sylow
g\'en\'eriques  pour  les  groupes  r\'eductifs  sur les corps finis'', {\sl
Math. Annalen \bf 292} (1992), 241--262.
\bibitem[Cabanes]{cabanes}  M.~Cabanes,  ``Unicit\'e du sous-groupe ab\'elien
distingu\'e maximal dans certains sous-groupes de Sylow'',
{\sl C.R.A.S. \bf 318} (1994), 889--894.
\bibitem[Enguehard]{En}   M.~Enguehard,  ``Sur les groupes de Sylow des groupes
 r\'eductifs finis'', unpublished notes of october 1992.
\bibitem[Gorenstein-Lyons]{GL} D.~Gorenstein and R.~Lyons,  ``The local
structure of finite groups of characteristic 2 type'', 
{\sl Memoirs of AMS \bf 42}, 1983.
\bibitem[Malle]{malle} G.~Malle, ``Height 0 characters of finite groups of Lie
type'', {\sl Representation theory \bf 11}(2007), 192--220.
\bibitem[Minkowski]{M} H.~Minkowski, ``Zur Theorie der positiven quadratischen
Formen'', {\sl J. Crelle \bf 101} (1887), 196-202.
\bibitem[Steinberg]{st}   R.~Steinberg,  ``Endomorphisms of linear algebraic
 groups'', {\sl Memoirs of the A.M.S. \bf 80}, 1965.
\end{thebibliography}
\end{document}